\documentclass[11pt,openany]{article}

\usepackage[pdftex]{graphicx}
 \DeclareGraphicsExtensions{.pdf,.jpeg,.png}
 
\usepackage{latexsym,amsmath,amssymb,amsfonts,eufrak}
\usepackage{epsfig,color}
\usepackage{hyperref}

\usepackage{amsthm, amscd, graphicx,makeidx,mathrsfs,float,capt-of}
\usepackage[all]{xy}
\usepackage{verbatim}
\usepackage{pifont}
\usepackage{color}

 \usepackage[all]{xy}

\usepackage[utf8]{inputenc} 
\usepackage[T1]{fontenc} 

\usepackage{palatino} 

\usepackage{mymacros}
\usepackage{geometry}
\usepackage{setspace}
\usepackage{amssymb,amsthm,amsmath,amsfonts}
\usepackage{tikz,pb-diagram}
\usetikzlibrary{arrows}

\usepackage{epigraph}

\newcommand{\de}{\Delta}
\newcommand{\Frobdim}{\mbox{{FrobDim}}}

\newcommand{\rt}{\rightarrow}

\newcommand{\al}{\alpha}
\newcommand{\be}{\beta}

\renewcommand {\k} {\Bbbk}

\newcommand{\ot}{\otimes}

\newcommand{\ti}{\times}

\newcommand{\id}{\ensuremath{\mathrm{id}}}

\newcommand{\h}{\mathrm{H}}

\theoremstyle{plain}
\newtheorem{theorem}{Theorem}[section]
\newtheorem{lemma}[theorem]{Lemma}
\newtheorem{prop}[theorem]{Proposition}

\theoremstyle{definition}
\newtheorem{defn}{Definition}[section]
\newtheorem{exmp}[theorem]{Example}

\newtheorem{coro}[theorem]{Corollary}

\theoremstyle{remark}
\newtheorem*{rem}{Remark}

\theoremstyle{conjecture}

\title{Nearly Frobenius Algebras}
\author{Ana Gonz\'alez\thanks{Instituto de Matem\'atica y Estad\'{\i}stica  ``Prof. Ing. Rafael Laguardia'',  Facultad de Ingenier\'{\i}a  Julio Herrera y Reissig, 565 CP11300. Montevideo, Uruguay.}, Ernesto Lupercio\thanks{Centro de Investigaciones y Estudios Avanzados,  Av. Instituto Polit\'ecnico Nacional 2508, Col. San Pedro Zacatenco, M\'exico D.F., M\'exico.}, Carlos Segovia\thanks{Instituto de Matem\'aticas de la UNAM, Campus Oaxaca. Oaxaca, M\'exico.} and Bernardo Uribe\thanks{Universidad del Norte, Barranquilla, Colombia.}}

\begin{document}
\maketitle
\begin{abstract}
\noindent In this introductory paper we  study nearly Frobenius algebras which are generalizations of the concept of a Frobenius algebra which appear naturally in topology: nearly Frobenius algebras have no traces (co-units). We survey the most basic foundational results and some of the applications they encounter in geometry, topology and representation theory. \end{abstract}
\section{Frobenius Algebras}

A Frobenius algebra is a pair $(A, \theta)$ where $A$ is a (graded, unital) $k$-algebra, and $\theta$ is a trace $\theta \colon A \to k$ (meaning that $\langle a | b \rangle := \theta(a\cdot b)$ is a non-degenerate bilinear form defined on $A$) \cite{B-N37,N39,N41}. Frobenius algebras play an important role in the theory of representations of groups and also in topology. To illustrate this, let us point out the two most basic examples of Frobenius algebras:

\begin{itemize}

	\item[a)] For a given finite group $G$, the group algebra $A := k[G]$ of a finite group together with the trace $\theta(\sum a_g g) := a_1$, is a Frobenius algebra.
	
	\item[b)] For a given $d$-dimensional compact, closed manifold $M$, the cohomology algebra $A=H^*(M,\R{})$ together with the trace $\theta(\omega) := \int_M \omega$, is also a Frobenius algebra.
	
\end{itemize}

Verifying the first assertion is immediate, while the second is equivalent to Poincaré duality.

To motivate the notion of a nearly Frobenius algebra (NFA form now on), let us consider two families of algebras that, while very similar to the two previous instances of Frobenius algebras, cannot be made into Frobenius algebras themselves:

\begin{itemize}

	\item[a')] The group algebra of the infinite group $\Integer$ of integers, $A:= k[\Integer]\cong k[t,t^{-1}]$, which coincides with the algebra of Laurent polynomials in one variable $t$. 
	
	\item[b')] The cohomology algebra $A:=H^*(M,k)$ for a non-compact manifold $M$.
	
\end{itemize}

In the first case, if we define $\theta(\sum a_i t^i)=a_0$, then 
$$\langle \sum a_i t^i | \sum b_i t^i \rangle = \sum a_i b_{-i},$$
whose value can easily take infinite values, for example, by taking $a_i=b_i=1$ for all $i$.

In the second case the situation is similar, for one can find a differential form defined on $M$ so that $\theta(\omega)=\int_M \omega=+\infty.$

As we will show below, both examples are NFAs, a concept that generalizes that of a Frobenius algebra in many natural examples. 

\section{Topological Quantum Field Theories}

To motivate the definition of an NFA let us revisit a famous reformulation of the definition of a Frobenius algebra, that is, the definition of a 2-dimensional topological quantum field theory (TQFT2 form now on) \cite{AtiyahQFT, MooreSegal}. First we need to define the category $\Cob$ of 2-dimensional cobordisms. The objects of $\Cob$ are non-negative integer numbers $n \geq 0$ thought of as disjoint unions of $n$ copies of the circle. The arrows are triples $(n,g,m)$ of non-negative integers thought of as topological 2-dimensional surfaces with $n$ incoming circles, $m$ outgoing circles and genus $g$. The case $n=0$ (respectively $m=0$) should be interpreted as a surface without incoming boundary components (resp. outgoing boundary components). 

 \begin{figure}[h!]
 \centering
  \includegraphics[width=0.5\textwidth]{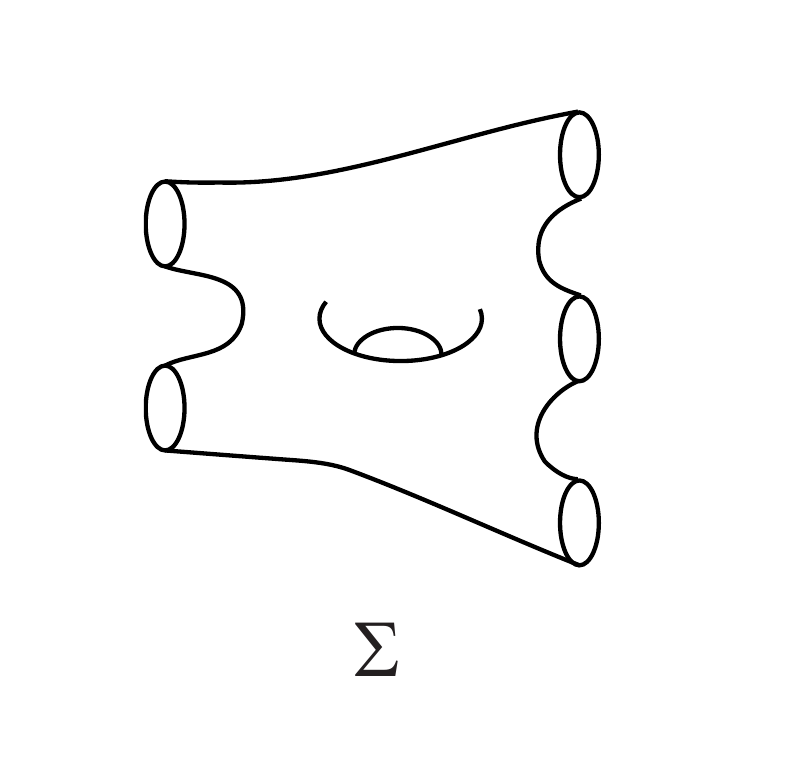}
 \caption{The morphism $(2,1,3)\colon 2 \to 3$, here $\Sigma=(2,1,3)$ in $\Cob$}\label{CobSigma}
 \end{figure}
 
 The composition law for morphisms in $\Cob$ is given by glueing surfaces: 
 $$(n',g',m'=n) \circ (n,g,m) := (n',g+g'+n-1,m).$$
 
  \begin{figure}[h!]
 \centering
  \includegraphics[width=0.4\textwidth]{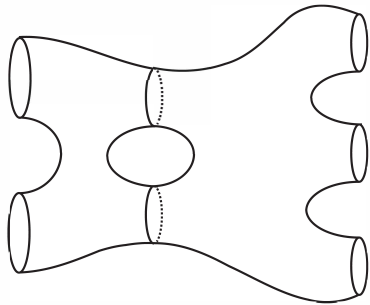}
 \caption{The composition $(2,0,2) \circ (2,0,3) = (2,1,3)$  in $\Cob$}\label{Glueing}
 \end{figure}

\begin{defn} 
A 2-dimensional TQFT is a functor $$Z\colon \Cob \to \Vect$$ from the category of 2-dimensional cobordisms to the category $\Vect$ of $k$-vector spaces so that:
$$Z(n+m) \cong Z(n) \otimes Z(m).$$
\end{defn}

Functoriality can be expressed by the matrix factorization:
$$ Z(n',g',m'=n) \circ Z(n,g,m) = Z(n',g+g'+n-1,n).$$
This is referred to as the glueing axiom.

We will also require $Z(1,0,1):A\to A$ to be the identity map. This requirement immediately implies the adjoint relation:
$$Z(n,g,m)^* = Z(m,g,n),$$ and the existence of a canonical isomorphism $A^*\cong A$. It also implies the finite dimensionality of $A$ as a $k$-vector space.

Clearly if we set once and for all $A:=Z(1)$, then $Z(n)=A^{\otimes n}$ is completely determined. It is also immediate to check that $\mu:=Z(1,0,2):A\otimes A \to A$ defines an associative product on $A$.The map $\theta:= Z(1,0,0) \colon A \to k$ is a trace on $A$ making it into a Frobenius algebra. To prove the associativity of $\mu$, for example, one uses the glueing axiom together with Figures \ref{assoc1} and \ref{assoc2}.

  \begin{figure}[h!]
 \centering
  \includegraphics[width=0.4\textwidth]{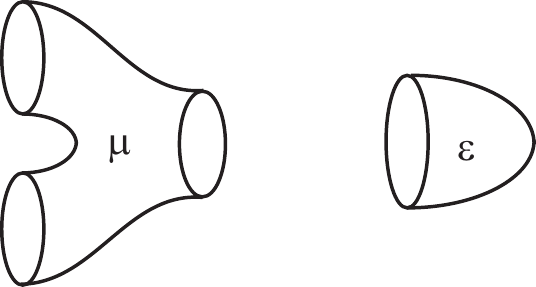}
 \caption{The multiplication $\mu \colon A \otimes A \to A$ and the trace map $\theta:A\to k$ for $A=Z(1)$.}\label{MultUnit}
 \end{figure}

 \begin{figure}[h!]
 \centering
  \includegraphics[width=0.2\textwidth]{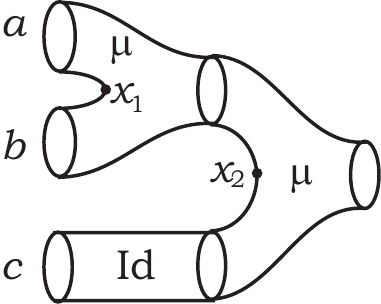}
 \caption{$(ab)c$}\label{assoc1}
 \end{figure}
 
 \begin{figure}[h!]
 \centering
  \includegraphics[width=0.2\textwidth]{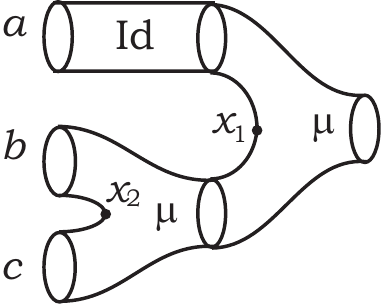}
 \caption{$a(bc)$}\label{assoc2}
 \end{figure}

 \begin{figure}[h!]
 \centering
  \includegraphics[width=0.4\textwidth]{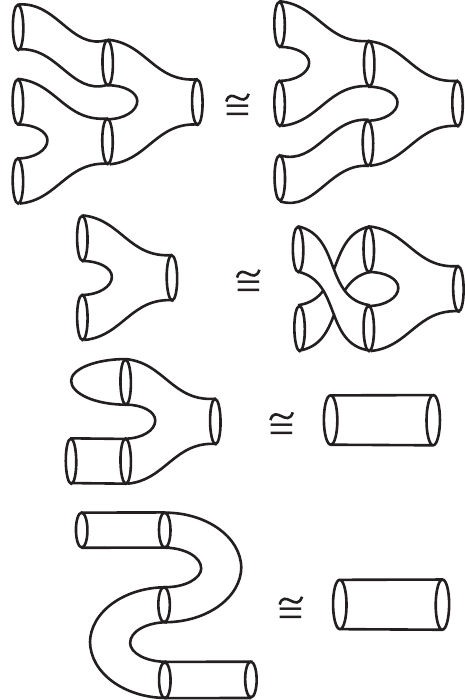}
 \caption{Each one of these cobordisms in $\Cob$ implies an algebraic property of $A$. From top to bottom: $a(bc)=(ab)c $, $ab=ba$, $1a=a$ and finally, the non-degeneracy of $\theta$}\label{AxiomsFrob}
 \end{figure}

It is a famous theorem that this construction defines an isomorphism of categories:
\begin{equation}
 F \colon \TQFT \to \Frob,
\end{equation} \label{Folk}
form the category of TQFT2 to the category of Frobenius algebras (e.g. \cite{MooreSegal}).

  So far, the algebras $A$ thus obtained could be non-commutative; for instance, the matrix algebra $A=\Mat_{n\times n}(k)$ together with $\theta(a_{ij}) = \sum_i a_{ii}$ is an example of this situation. From now on, we will add the braiding morphism $\tau : 2 \to 2$ in $\Cob$, having the effect at the level of vector spaces of adding a structure map to $A$ of the form, $Z(\tau)\colon A\otimes A \to A \otimes A$, given by $Z(\tau)(a\otimes b ) = (-1)^{|a||b|} b \otimes a$. In other words, all the algebras $A$ that we will consider in the following will be super-commutative (second row of Figure \ref{AxiomsFrob}).
  
To proof of the isomorphism of categories $\TQFT \cong \Frob$ one has to show that given a Frobenius algebra $(A,\theta)$, one can reconstruct the linear map $Z(n,g,m)$ for any given $\Sigma=(n,g,m)$. To do this one first decomposes by induction $\Sigma$ into the elementary cobordisms $(2,0,1)$, $(0,0,1)$, $(1,0,1)$, $(1,0,0)$ and $(1,0,2)$. That such a decomposition exists is easy to prove using Morse theory, but one must bear in mind that it is far from unique. 

\begin{figure}[h!]
 \centering
  \includegraphics[width=0.5\textwidth]{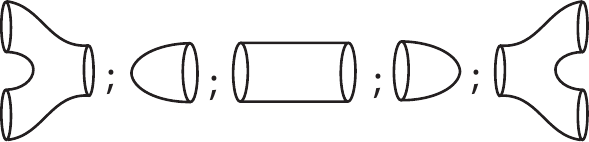}
 \caption{The elementary cobordisms $(2,0,1)$, $(0,0,1)$, $(1,0,1)$, $(1,0,0)$ and $(1,0,2)$. They correspond under the functor $Z$ to the maps $\mu=Z(2,0,1)$, $u=Z(0,0,1)$, $\Id_A=Z(1,0,1)$, and their duals, $\theta=Z(1,0,0)$ and $\Delta=Z(1,0,2)$.}\label{ElementaryCobs}
 \end{figure}

Notice that the multiplicative unit element $1 \in A$ can be thought of as a map $u: k \to A$ written as $u(\lambda) = \lambda \cdot 1 \in A$, and therefore as $u=Z(0,0,1)$. The topological interpretation of the unital property of $1\in A$ is depicted in the third row of Figure \ref{AxiomsFrob}.

Using the canonical isomorphism $A\cong A^*$ we can write maps $\theta=Z(1,0,0)=Z(0,0,1)^*=u^*$ and $\Delta=Z(1,0,2)=Z(2,0,1)^*=\mu^*$. It is reasonable to expect $\Delta$ to be a co-product with co-unit $\theta$, and indeed this is the case.

 A co-associative, co-commutative, co-unital co-product $\Delta \colon A \to A\otimes A$ on a $k$-vector space $A\in \Vect$ is the same thing as an associative, commutative, unital product on the corresponding element $A \in \Vect^{\op}$ of the opposite category. In other words, for example, the diagram that defines co-associativity for $A$:
$$
\xymatrix{
A\ar[r]^{\Delta}\ar[d]_{\Delta}& A\otimes A\ar[d]^{\Delta \otimes 1}\\
A\otimes A\ar[r]_{1\otimes \Delta}&A\otimes A\otimes A}
$$
is obtained by inverting the arrows of the usual diagram for associativity of a product.

We will write $\Delta(x)=\sum x_1\otimes x_2$ rather than $\Delta(x)=\sum_i (x_1)_i \otimes (x_2)_i$ omitting at every chance the summation indices. We prefer the use of diagrams, for even with this notational simplification, the co-associtivity property for $\Delta$ looks complicated in explicit form:
$$
 (\Delta \otimes1)\bigl(\Delta (x)\bigr)  =  \sum x_{11}\otimes x_{12}\otimes x_2  =\sum x_1\otimes x_{21}\otimes x_{22}=(1\otimes\Delta)\bigl(\Delta(x)\bigr).
$$
There is a third way of representing this co-associativity: using topological cobordisms in $\Cob$. Just as the usual associativity of $A$ can be interpreted as the equality between the two different factorizations of $Z(3,0,1)$ given by the sliding the saddle point $x_1$ past the saddle point $x_2$ in Figures \ref{assoc1} and \ref{assoc2}, the co-associativity of $\Delta$ is likewise proved by means of Figure \ref{Coassoc}: it amounts to the ability to slide the saddle point $x_2$ past the saddle point $x_1$.

\begin{figure}[h!]
 \centering
  \includegraphics[width=0.5\textwidth]{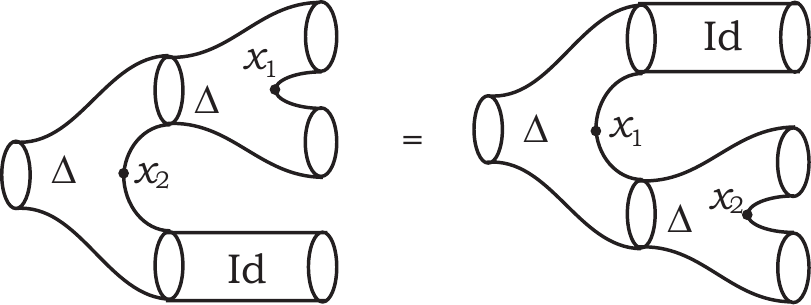}
 \caption{Co-associativity for $A$: $
 (\Delta \otimes1)\bigl((\Delta (x))\bigr)  =  (1\otimes\Delta)\bigl(\Delta(x)\bigr).
$}\label{Coassoc}
 \end{figure}

From the discussion above, we have managed to associate an operator $Z(n,g,m)$ to every decomposition of $\Sigma=(n,g,m)$ into elementary cobordisms, but we do not know that this operator does not depend on the decomposition; in fact, it doesn't. The proof of this independence can be divided into two steps: the first step being an algebraic lemma, and the second step being an argument in combinatorial topology. The algebraic lemma is as follows (cf. \cite{Abrams}):

\begin{lemma}\label{AlgLemma}
Given a fixed k-algebra $(A,\mu,1)$, there is a one-to-one correspondence between Frobenius structures $\theta\colon A\to k$ on $A$, and co-associative, co-commutative, co-unital co-products $\Delta \colon A \to A\otimes A$ that happen to be $A$-bi-module maps.
\end{lemma}

The bi-module action on $A\otimes A$ on decomposable elements is described by:
$$ a(b\otimes b') a' = (ab) \otimes (b'a').$$

To prove the lemma, it is very easy to see that having $\Delta$ its co-unit map $A\to k$ is precisely $\theta$: the non-degeneracy of the trace is an immediate consequence of the   following two commutative diagrams relating the product and the coproduct: 

\begin{equation*}
	\xymatrix@C=1cm@R=1cm{
		{A}\otimes{A}\ar[r]^\mu\ar[d]_{1\otimes\de}&{A}\ar[d]^\de&{A}\otimes{A}\ar[r]^{\mu}\ar[d]_{\de\otimes 1}&{A}\ar[d]^{\de}\\
		{A}\otimes{A}\otimes{A}\ar[r]_{\mu\otimes
			1}&{A}\otimes{A}&{A}\otimes{A}\otimes{A}\ar[r]_{1\otimes \mu}&{A}\otimes{A} }
\end{equation*}

Both diagrams can be better encoded in figures \ref{FrobIden1} and \ref{FrobIden2}.

\begin{figure}[h!]
	\centering
	\includegraphics[width=0.3\textwidth]{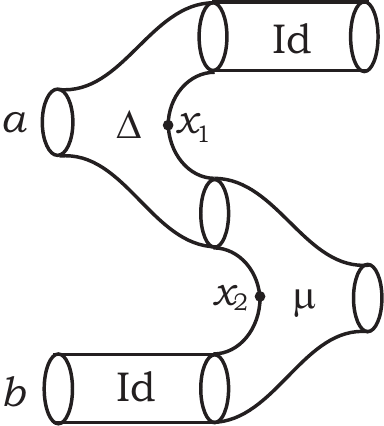}
	\caption{$a\otimes b\longmapsto\Delta(a)b$}\label{FrobIden1}
\end{figure}

\begin{figure}[h!]
	\centering
	\includegraphics[width=0.3\textwidth]{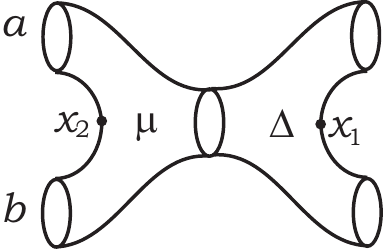}
	\caption{$ a\otimes b\longmapsto\Delta(ab)$}\label{FrobIden2}
\end{figure}

What these pictures (equivalently, the above commutative diagrams) tell us  is that we can ``slide'' the critical point $x_1$ past $x_2$ exchanging their positions. In any case, by capping off with a unit a input boundary and with a co-unit an output boundary component, we can readily imply the non-degeneracy of the trace.  

Conversely if we have a non-degenerate trace, as a consequence we have that $A\cong A^*$ and therefore we can simply define $\Delta$ as the dual of the product, $\mu^*$.  

Lemma  \ref{AlgLemma} above is then equivalent to the assertion that we are allowed to slide critical points past one another in a given cobordism, changing in the process the corresponding pair-of-pants decomposition. We make this more explicit shortly. 

To complete the proof of the fact that there is an equivalence of categories  $\Frob \to  \TQFT$, we need to use Cerf-Morse's theory \cite{Cerf}. Remember, this will be proved once we know that the linear mapping associated to a cobordism do not depend on the particular pair-of-pants decomposition used to define it.

Let us briefly review the consequences of Cerf's theory that we will be using. First we fix the topological surface  with boundary $\Sigma$. We will consider the (connected) space $\mathcal{M}$ of pairs $(\Sigma, f)$ where $f:\Sigma \to \mathbb{R}$ is a Morse function so that $f$ restricted to the input boundary components has constant value 0, and restricted to the output boundary components has constant value 1. To such a pair we can associate a well-defined pants-decomposition of $\Sigma$ by cutting it up in between critical points (such points will be labeled $x_1,\ldots, x_r$, and will be ordered by the value that $f$ takes on them). This setting is depicted in figure \ref{Canonical}  below.

\begin{figure}[h!]
	\centering
	\includegraphics[width=0.5\textwidth]{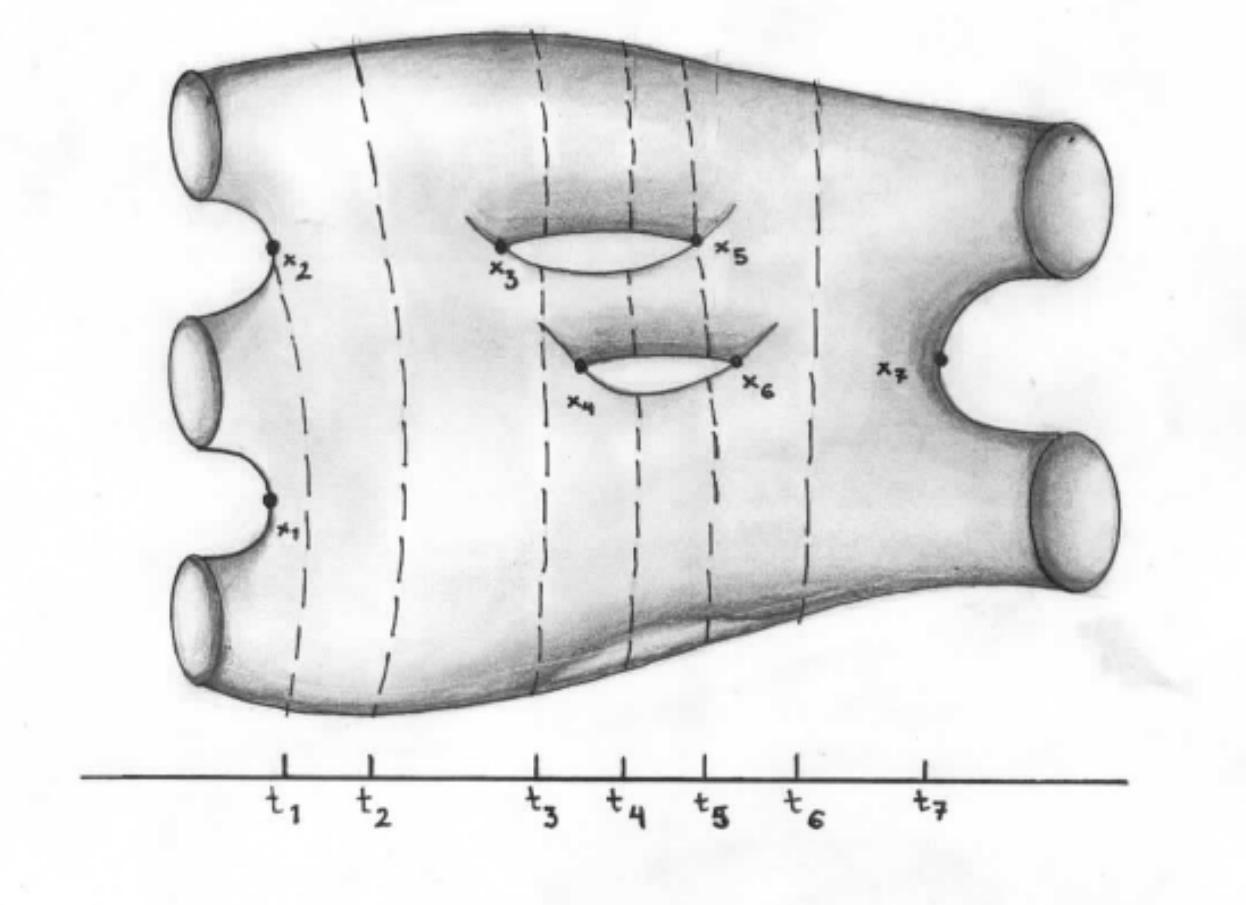}
	\caption{To every Morse function $f$ with distinct critical points we associate a decomposition of the surface $\Sigma$. }\label{Canonical}
\end{figure}

As we vary $f$ in $\mathcal{M}$, the pants-decomposition changes: this happens  as any two consecutive critical points $x_i,x_{i+1}$, $f(x_i)<f(x_{i+1})$, cross a wall  (such wall defined by the condition that $f$ takes the same value on both points $f(x_i) = f(x_{i+1})$), and then, as $f$ changes, they exchange places $x_i \leftrightarrow x_{i+1}$, $f(x_{i+1})<f(x_i)$. It is a remarkable consequence of Cerf's theory that this procedure \emph{connects all possible pair-of-pants decompositions}; all of them can be reached by exchanging critical points. This happens because the connected space of Morse functions is divided into thick chambers of functions for which $f(x_i) \neq f(x_j)$ for all pairs $i \neq j$; the complement of this generic condition forms walls, and moreover, it is enough to cross a finite number of walls to get from any $f_0$ to any other $f_1$ in the space $\mathcal{M}$, for example every decomposition can always be taken into a "canonical form" as in figure \ref{Crossing}. 

\begin{figure}[h!]
	\centering
	\includegraphics[width=0.9\textwidth]{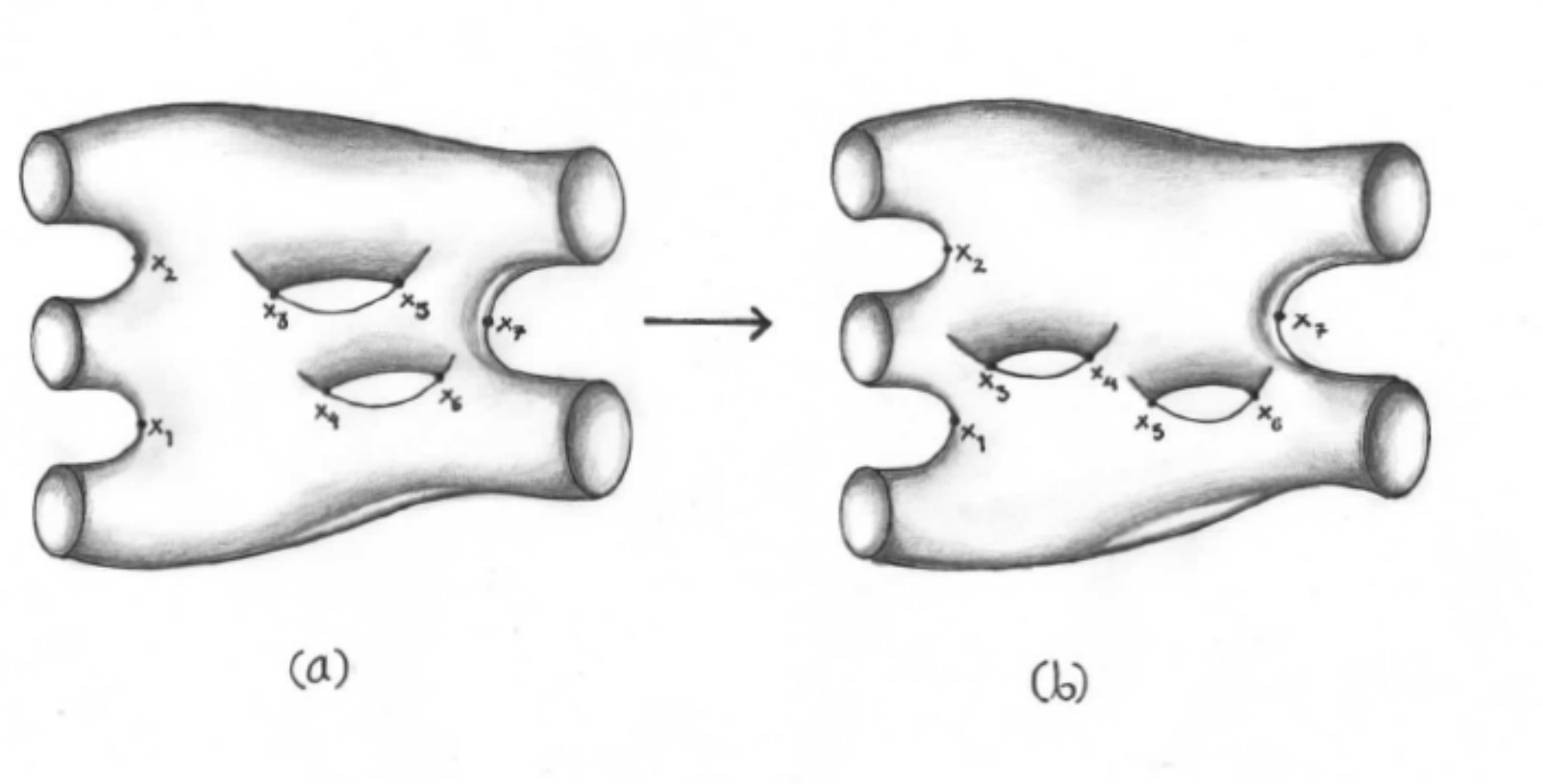}
	\caption{Two critical points $x_4$ and $x_5$ cross as we move $f$ in $\mathcal{M}(Z)$, taking $f_a$ into its final canonical form $f_b$. The algebraic change between the decomposition (a) and (b) is calculated by Lemma \ref{AlgLemma}. }\label{Crossing}
\end{figure}

As we have seen, given a TQFT, defining a Frobenius algebra structure on the vector space $A$ associated to the connected boundary circle is very straighforward. Conversely, to associate a TQFT to a Frobenius algebra in a functorial way we proceed as follows: given a Frobenius algebra $A$ and a cobordism $\Sigma$ (with $n$ incoming circles and $m$ outgoing circles) we pick any $f$ whatsoever in $\mathcal{M}$, and then we use the correspoding pair-of-pants decomposition induced in $\Sigma$ to define the multi-linear operator $A^n \to A^m$. Finally because of the algebraic lemma \ref{AlgLemma} and Cerf's theory we immediately conclude that the assigment is well defined (independent of the decomposition) and functorial, finishing the proof that that there is an equivalence of categories  $\Frob \to  \TQFT$.

\section{Nearly Frobenius algebras}

We are ready to define the main object of this paper: nearly Frobenius algebras.

\begin{defn}
	A \emph{nearly Frobenius algebra} is a pair $((A,\cdot),\Delta)$ of an algebra (over a field $k$) together with a co-associative (possibly not co-unital) co-multiplication $\Delta: A\otimes A \to A$ that is a bi-module map form $A \otimes A$ to $A$.  
\end{defn}

The very last property of the definition can be succinctly defined by the \emph{Frobenius conditions} (cf. figures \ref{FrobIden1} and \ref{FrobIden2}):
$$ a\Delta(b) = \Delta(ab) = \Delta(a)b.$$

Some examples are in order.

\begin{exmp} First, every Frobenius algebra $(A,\mu,\theta)$ is a nearly Frobenius algebra. For $\theta$ induces an isomorphism $A\cong A^*$ and the dual of the multiplication, $\Delta := \mu^*$, is the co-multiplication of the nearly Frobenius algebra structure. It is straighforward to check the Frobenius identities for this example in view of \ref{AlgLemma}. 
\end{exmp}

Not every nearly Frobenius algebra is a Frobenius algebra, the easiest example is the algebra for which the multiplication $\mu$ is identically zero: the identically zero co-multiplication $\Delta := 0$ makes it into a nearly Frobenius algebra. A more interesting example is below, while Frobenius algebras have an essentially unique Frobenius trace, nearly Frobenius algebras can have a whole space of nearly Frobenius co-products.

\begin{exmp}
	Consider $A$ to be a truncated polynomial algebra in one variable
	$K[x]/x^{n+1}$m for $K$ a field. We can endow $A$ with many nearly Frobenius algebra structures.
	
	We will fix the basis $\{1,x,\dots,x^n\}$ for $A$.  Any  $k$-linear map $\Delta:A \to  A\otimes A$ evaluated at
	$1$ takes the form:
	$$\Delta(1)=\sum_{i,j=1}^n a_{ij} x^i\otimes x^j.$$
	To make this map into an $A$-bimodule morphism we need the following to hold:
	\begin{equation}\label{equation}
	\Delta(x^k)=(x^k\otimes 1)\Delta(1)=\Delta(1)(1\otimes
	x^k),\quad\forall\;k\in\{0,\dots,n\}.\end{equation} 
	
	Specializing to $k=1$ we get
	$$\sum_{i,j=1}^na_{ij}x^{i+1}\otimes x^j=\sum_{i,j,=1}^na_{ij}x^i\otimes x^{j+1}.$$
	Which occurs when $a_{0,j-1}=0$, $j=1,\dots,n$; $a_{i-1,0}=0$,
	$i=1,\dots,n$ and $a_{i,j-1}=a_{i-1,j}$ otherwise.  Then
	$$\Delta(1)=\sum_{k=0}^na_{kn}\left(\sum_{i+j=n+k}x^i\otimes x^j\right).$$
	We will define $a_k:=a_{kn}$. To conclude  that $\Delta$ is an $A$-bimodule morphism, we need to
	show that $\Delta(x^k)=\bigl(x^k\otimes 1\bigr)\Delta(1)=\Delta(1)\bigl(1\otimes
	x^k\bigr)$.
	
	$$\begin{array}{rcl}
	\Delta(1)\bigl(1\otimes x^l\bigr) & = & \displaystyle{\sum_{k=0}^na_{k}\left(\sum_{i+j=n+k}x^i\otimes x^j\right)\bigl(1\otimes x^l\bigr)}
	=  \displaystyle{\sum_{k=0}^na_{k}\left(\sum_{i+j=n+k}x^i\otimes x^{j+l}\right)} \\
	& = & \displaystyle{\sum_{k=0}^na_{k}\left(\sum_{i+m=n+k+l}x^i\otimes x^m\right)}
	=  \displaystyle{\sum_{k=0}^na_{k}\left(\sum_{r+m=n+k}x^{r+l}\otimes x^m\right)} \\
	& = & \displaystyle{\bigl(x^l\otimes 1\bigr)\sum_{k=0}^na_{k}\left(\sum_{r+m=n+k}x^r\otimes x^m\right)}
	=  \displaystyle{\bigl(x^l\otimes 1\bigr)\Delta(1)}.
	\end{array}
	$$
	
		Next we verify the coassociativity axiom: Let $x^l\in A$
	with $l\geq 0$.
	$$
	\begin{array}{rcl}
	\bigl(\Delta\otimes 1\bigr)\bigl(\Delta\bigl(x^l\bigr)\bigr)  & = & \displaystyle{\bigl(\Delta\otimes 1\bigr)\left(\sum_{k=0}^na_{k}\left(\sum_{i+j=n+k+l}x^i\otimes x^j\right)\right)  = \sum_{k=0}^na_{k}\left(\sum_{i+j=n+k+l}\Delta\bigl(x^i\bigr)\otimes x^j\right)} \\
	& = & \displaystyle{\sum_{k,m=0}^na_ka_m\left(\sum_{i+j=n+k+l}\sum_{r+s=n+m+i}x^r\otimes x^s\otimes x^j\right)} \\
	& = & \displaystyle{\sum_{k,m=0}^na_ka_m\left(\sum_{r+s+j=2n+m+k+l}x^r\otimes x^s\otimes x^j\right)} \\
	\bigl(1\otimes\Delta\bigr)\bigl(\Delta\bigl(x^l\bigr)\bigr)  & = &  \displaystyle{\bigl(1\otimes\Delta)\left(\sum_{k=0}^na_{k}\left(\sum_{i+j=n+k+l}x^i\otimes x^j\right)\right)  = \sum_{k=0}^na_{k}\left(\sum_{i+j=n+k+l}x^i\otimes \Delta\bigl(x^j\bigr)\right)}\\
	& = & \displaystyle{\sum_{k,m=0}^na_ka_m\left(\sum_{i+j=n+k+l}\sum_{r+s=n+m+j}x^i\otimes x^r\otimes x^s\right)} \\
	& = & \displaystyle{\sum_{k,m=0}^na_ka_m\left(\sum_{r+s+j=2n+m+k+l}x^r\otimes x^s\otimes x^j\right)}. \\
	\end{array}
	$$
	
	Therefore $\bigl(A,\Delta\bigr)$ is a nearly-Frobenius algebra. Moreover, any coproduct $\Delta$ is a linear combination
	of the coproducts $\Delta_k$ defined by the formula:
	$$\Delta_k\bigl(x^l\bigr)=\sum_{i+j=n+k+l}x^i\otimes x^j,\quad\mbox{for}\; k\in\{0,\dots,n\}$$

	Notice that $\Delta_0$ is the Frobenis coproduct of $A$ where the trace
	map $\theta:A\to\mathbb{C}$ is given by
	$\theta\bigl(x^i\bigr)=\delta_{i,n}$. The remaining coproducts,
	$\Delta_k$ $k\neq 0$, do not come from a Frobenius algebra structure: for $k\neq 0$, it doesn't exist a trace map $\theta:A\to \mathbb{C}$ such that
	$\bigl(A, \de_k,\theta\bigr)$ is a Frobenius algebra. Indeed, otherwise:
	$$m(\theta\otimes 1)\bigl(\de_k\bigl(x^l\bigr)\bigr)=\sum_{i+j=n+k+l}\theta\bigl(x^i\bigr)x^j,$$
	with $j>l$, so $m(\theta\otimes 1)\de_k\bigl(x^l\bigr)\neq x^l.$
\end{exmp}

In view of the previous example, we make a definition.

\begin{defn} The \emph{Frobenius dimension} $\Frobdim(A)$ of an algebra $(A,\mu)$ is the dimension of the moduli variety $\mathcal{N}(A)$ of nearly Frobenius structures $\Delta$ on $A$ which are compatible with $\mu$.
\end{defn}  

In the previous example the Frobenius dimension of $A=\mathbb{C}[x]/x^{n+1}$ is $n+1$ and coincides with the dimension of $A$ as a vector space over $\mathbb{C}$.

While Frobenius algebras are bound to be finite dimensional, nearly Frobenius algebras are free to be infinite dimensional.

\begin{exmp}
	Let us consider $A$ to be the algebra  $\mathbb{C}\bigl[\bigl[x,x^{-1}\bigr]\bigr]$
	of formal Laurent series. The coproducts given by:
	$$\Delta_{j}\bigl(x^{i}\bigr)=\sum_{k+l=i+j} x^{k} \otimes x^{l}$$
	 define infinitely many nearly Frobenius structures on $A$ that do not come
	from a Frobenius structure.
\end{exmp}

\begin{theorem}\label{t1}
	If $\bigl(A,\de_1\bigr)$ and $\bigl(B,\de_2\bigr)$ are
	nearly-Frobenius algebras then $\bigl(A\ot B, \de\bigr)$  is a
	nearly-Frobenius algebra where
	$$\de=(1\ot\tau\ot 1)\circ\bigl(\de_1\ot\de_2\bigr),\quad\mbox{(here}\; \tau\;\mbox{is the transposition).}$$
\end{theorem}
\begin{proof}
	The map $\de$ is coassociative because the external diagram is
	commutative since the internal diagrams commute:
	$$\xymatrix@R=1.2cm@C=1.5cm{
		A\ot B\ar[r]^{\de_1\ot\de_2}\ar[d]_{\de_1\ot\de_2}&A\ot A\ot B\ot
		B\ar[r]^{1\ot\tau\ot 1}\ar[d]^{\de_1\ot 1\ot\de_2\ot 1}& (A\ot
		B)\ot(A\ot B)\ar[d]^{\de_1\ot\de_2\ot 1\ot 1}\\
		A\ot A\ot B\ot B\ar[r]^{1\ot\de_1\ot 1\ot\de_2}\ar[d]_{1\ot\tau\ot
			1}&A\ot A\ot A\ot B\ot B\ot B\ar[r]^{1\ot\tau\ot
			1}\ar[d]^{1\ot\tau\ot 1}&A\ot A\ot B\ot B\ot A\ot B\ar[d]^{1\tau\ot
			1\ot 1\ot 1}\\
		A\ot B\ot B\ot A\ot B\ar[r]_{1\ot 1\ot 1\ot\de_1\ot\de_2}&A\ot B\ot
		A\ot A\ot B\ot B\ar[r]_{1\ot \tau \ot 1\ot 1\ot 1 \ot 1}&A\ot B\ot A\ot
		B\ot A\ot B
	}$$
	The linear map $\de$ satisfies the Frobenius identities because the
	next external diagram is commutative using that the internal
	diagrams commute:
	$$\xymatrix@R=1.2cm@C=2cm{
		(A\ot B)\ot(A\ot B)\ar[r]^{1\ot\tau\ot 1}\ar[d]_{\de_1\ot\de_2\ot
			1}&A\ot A\ot B\ot B\ar[r]^{m_1\ot m_2}\ar[d]^{\de_1\ot 1\ot \de_2\ot
			1}&
		(A\ot B)\ar[d]^{\de_1\ot\de_2}\\
		A\ot A\ot B\ot B\ot A\ot B\ar[r]^{1\ot\tau\ot 1}\ar[d]_{1\ot\tau\ot
			1\ot 1}&A\ot A\ot A\ot B\ot B\ot B\ar[r]^{1\ot m_1\ot 1\ot
			m_2}\ar[d]^{1\ot\tau\ot 1\ot 1}&A\ot A\ot B\ot B\ar[d]^{1\tau\ot
			1}\\
		A\ot B\ot A\ot B\ot A\ot B\ar[r]_{1\ot 1\ot \tau\ot 1}&A\ot B\ot
		A\ot A\ot B\ot B\ar[r]_{1\ot m_1\ot m_2}&A\ot B\ot A\ot B
	}$$
\end{proof}

Let be $\bigl(A, \de\bigr)$ a nearly-Frobenius algebra.
\begin{defn}
	A linear subspace $J$ in $A$ is called a \emph{nearly-Frobenius
		ideal } if
	\begin{enumerate}
		\item[(a)]$J$ is an ideal of $A$ and
		\item[(b)] $\de(J)\subset J\ot A + A\ot J$.
	\end{enumerate}
\end{defn}

\begin{prop}\label{p1}
	Let be $\bigl(A,\de\bigr)$ a nearly-Frobenius algebra, $J$ a
	nearly-Frobenius ideal  and $p:A\rightarrow A/J$ the natural
	projection. Then, $A/J$ admits a unique nearly-Frobenius structure
	such that $p$ is a coalgebra morphism.
\end{prop}
\begin{proof}
	Since $(p\otimes p)\Delta(J)\subset(p\otimes p)\bigl(J\otimes
	A+A\otimes J\bigr)=0$, it follows that there exists a unique morphism of
	vector spaces $$\exists! \quad\overline{\Delta}:A/J\rightarrow
	A/J\otimes A/J$$ for which the diagram:
	$$\xymatrix{A\ar[r]^p\ar[d]_{\Delta}&A/J\ar[d]^{\overline{\Delta}}\\
		A \ar[r]_(.3){p\otimes p}& A/J\otimes A/J}$$ is commutative. This
	map is defined by $\overline{\Delta}(\overline{a})=\sum
	\overline{a_1}\otimes\overline{a_2}$ where $\overline{a}=p(a)$, i.e.
	$\overline{\Delta}=(p\otimes p)\circ \Delta$.
	
	The fact that   $\bigl(\overline{\Delta}\otimes
	1\bigr)\overline{\Delta}(\overline{a})=\bigl(1\otimes\overline{\Delta}\bigr)\overline{\Delta}(\overline{a})=\sum\overline{a_1}\otimes\overline{a_2}\otimes\overline{a_3}$
	follows immediately from the commutativity of the diagram.
	
	The coproduct is a bimodule morphism, for:
	$$\xymatrix{A/J\otimes A/J\ar[r]^{\overline{m}}\ar[d]_{\overline{\Delta}\otimes 1}&A/J\ar[d]^{\overline{\Delta}}\\
		A/J\otimes A/J\otimes A/J\ar[r]_(.6){1\otimes\overline{m}}&
		A/J\otimes A/J}\quad
	\xymatrix{A/J\otimes A/J\ar[r]^{\overline{m}}\ar[d]_{1\otimes\overline{\Delta}}&A/J\ar[d]^{\overline{\Delta}}\\
		A/J\otimes A/J\otimes A/J\ar[r]_(.6){\overline{m}\otimes 1}&
		A/J\otimes A/J}$$
	
	Notice that  $\Delta(a)=\sum a_1\otimes a_2$, $\Delta(b)=\sum
	b_1\otimes b_2$.
	
	$\overline{\Delta}\overline{m}(\overline{a}\otimes\overline{b})=\overline{\Delta}(p(ab))=(p\otimes
	p)\Delta(ab)=(p\otimes p)\bigl((1\otimes m)(\Delta\otimes
	1)(a\otimes b)\bigr)$ $=(p\otimes p)\bigl(\sum a_1\otimes
	a_2b\bigr)=\sum \overline{a_1}\otimes\overline{a_2b}$.
	
	Also, $(1\otimes\overline{m})(\overline{\Delta}\otimes
	1)(\overline{a}\otimes\overline{b})=(1\otimes\overline{m})\bigl(\sum\overline{a_1}\otimes\overline{a_2}\otimes\overline{b}\bigr)=\sum\overline{a_1}\otimes\overline{a_2b}$.
	Making the first diagram commutative.
\end{proof}

\section{Nearly Frobenius Semisimple Algebras}
Whenever $A$ is a semisimple algebra we can fully classify all compatible nearly Frobenius structures on $A$.\\
\subsection{Non-commutative fields}
	Let $\k$ be a non-commutative field with center $Z(\k)$.\\
	If a linear map $\de:\k\rt\k$ satisfies the Frobenius identities
	then we have: $$\de(x)=\de(1)(1\ot x)=(x\ot 1)\de(1),\quad\forall\;x\in\k.$$
	Writing the coproduct in an anzats: $\de(1)=a1\ot 1=a\ot 1$, with $a\in\k$ we have:
	$$\de(x)=a\ot x= ax\ot 1= xa\ot 1 \Leftrightarrow ax=xa \Leftrightarrow a\in Z(\k).$$
	This coproduct is coassociative:
	$$\begin{array}{rclclcl}
	(\de\ot 1)\de(x) & = & \de(1)\ot ax & = & a\ot 1\ot ax & = & a^2x\ot 1\ot 1\\
	(1\ot \de)\de(x) & = & ax \ot\de(1) & = & ax\ot a\ot 1 & = & axa\ot 1\ot 1
	\end{array}
	$$
	Since $a\in Z(\k)$ we have that $a^2x=axa$, and then $(\de\ot 1)\de(x)=(1\ot \de)\de(x)$, $\forall x\in\k$.\\
	Thus, the algebra $A:=\k$ is a nearly-Frobenius algebra and we have as many nearly-Frobenius structures on $\k$ as elements in the center of $\k$. \\
	Note that all  structures come from  Frobenius structures where the
	trace map is $\theta:\k\rt\k$ is given by $\theta(1)=1$.

\subsection{Matrix algebras}
	More generally let $A$ be the matrix algebra $M_{n\times n}(\k)$, with $\k$ a
	commutative field. We write the canonical basis of $A$ as:
	$\bigl\{E_{ij}:\;i,j=1,\dots ,n\bigr\}$, where
	$E_{ij}=\bigl(e_{kl}\bigr)_{kl}$ with
	$$e_{kl}=\left\{\begin{array}{lr}
	1 & \mbox{if}\; k=i, l=j \\
	0 & \mbox{in other case}
	\end{array}
	\right..$$ Notice that: $E_{ij}E_{kl}=\left\{\begin{array}{lr}
	E_{il} & \mbox{if}\; j=k \\
	0 & \mbox{in other case}
	\end{array}
	\right..$ In particular $E_{ii}E_{ij}=E_{ij}$ and
	$E_{ij}E_{jj}=E_{ij}$, then:
	$$\de\bigl(E_{ij}\bigr)=\de\bigl(E_{ij}\bigr)\bigl(1\ot E_{jj}\bigr)=\bigl(E_{ii}\ot 1\bigr)\de\bigl(E_{ij}\bigr)$$
	and $$\de\bigl(E_{ij}\bigr)=\de\bigl(E_{ii}\bigr)\bigl(1\ot
	E_{ij}\bigr)=\bigl(E_{ij}\ot 1\bigr)\de\bigl(E_{jj}\bigr).$$ The
	the last equations immeditely imply that: $$\de\bigl(E_{ij}\bigr)=\sum_{k,l=1}^n
	a_{kl}^{ij}E_{ik}\ot E_{lj}=\sum_{k,l=1}^n a_{kl}^{ii}E_{ik}\ot
	E_{lj}=\sum_{k,l=1}^n a_{kl}^{jj}E_{ik}\ot E_{lj},$$ and then,
	$a_{kl}^{ij}=a_{kl}^{ii}=a_{kl}^{jj}$, for all $k,l=1\dots ,n.$ This in turn implies: $$\de\bigl(E_{ij}\bigr)=\sum_{k,l=1}^n
	a_{kl}E_{ik}\ot E_{lj},\quad\forall\;i,j.$$
	
	Finally, we need to check that this coproduct is coassociative:
	$$\begin{array}{lclcl}
	(\de\ot 1)\de\bigl(E_{ij}\bigr) & = &\displaystyle{ \sum_{k,l=1}^na_{kl}\de\bigl(E_{ik}\bigr)\ot E_{lj}} & = & \displaystyle{\sum_{k,l=1}^n\sum_{r,s=1}^na_{kl}a_{rs}E_{ir}\ot E_{sk}\ot E_{lj}} \\
	(1\ot\de)\de\bigl(E_{ij}\bigr) & = & \displaystyle{\sum_{r,s=1}^na_{rs}E_{ir}\ot\de\bigl(E_{sj}\bigr)} & = & \displaystyle{\sum_{r,s=1}^n\sum_{k,l=1}^na_{rs}a_{kl}E_{ir}\ot E_{sk}\ot E_{lj} }
	\end{array}
	$$
	As $\k$ is commutative, we have that $(\de\ot
	1)\de\bigl(E_{ij}\bigr)=(1\ot\de)\de\bigl(E_{ij}\bigr)$.
	
	Note that $M_{n\times n}(\k)$ admits $n\times n$ independent
	coproducts, one for each $a_{kl}$, namely:
	$$\de\bigl(E_{ij}\bigr)=\sum_{k,l=1}^na_{kl}\de_{kl}\bigl(E_{ij}\bigr),\;\mbox{where}\;\de_{kl}\bigl(E_{ij}\bigr)=E_{ik}\ot E_{lj}.$$

\subsection{Cyclic algebras}

	Let $G$ be a cyclic finite group. The group  $\k[G]$ is a
	nearly-Frobenius algebra. A basis, as vector space, of $\k[G]$ is
	$\bigl\{g^i:\;i=1,\dots,n\bigr\}$ where $|G|=n$. As before, if we
	determine the value of the coproduct in the unit of the group, we
	have the value over all element of the algebra.
	
	A general expression of $\de(1)$ is:
	$$\displaystyle{\de(1)=\sum_{i,j=1}^n\al_{ij}g^i\ot g^j}.$$ Using
	that $\de\bigl(g^k\bigr)=\de(1)\bigl(1\ot g^k\bigr)=\bigl(g^k\ot
	1\bigr)\de(1)$, we have that:
	$$\sum_{i,j^1}^n\al_{ij}g^{k+i}\ot g^j=\sum_{i,j=1}^n\al_{ij}g^i\ot g^{j+k}.$$
	then, $\al_{i-kj}=\al_{ij-k}$, also $\al_{1j-1}=\al_{nj}$ and
	$\al_{in}=\al_{i-11}$. This permit us to express the coproduct as:
	$$\de(1)=\sum_{i=2}^n\al_i\left\{\sum_{k=1}^{i-1}g^k\ot g^{i-k}+\sum_{k=i}^ng^k\ot g^{n+i-k}\right\}.$$
	This implies that:
	$$
	\begin{array}{rcl}
	\de\bigl(g^k\bigr) & = & \displaystyle{\de(1)\bigl(1\ot g^k\bigr) }\\
	& = & \displaystyle{\sum_{i=2}^n\al_i\left\{\sum_{k=1}^{i-1}g^{k+l}\ot g^{i-k}+\sum_{k=i}^ng^{k+l}\ot g^{n+i-k}\right\}} \\
	\de\bigl(g^k\bigr) & = & \displaystyle{\bigl(g^k\ot 1\bigr)\de(1) }\\
	& = & \displaystyle{\sum_{i=2}^n\al_i\left\{\sum_{k=1}^{i-1}g^k\ot g^{i+l-k}+\sum_{k=i}^ng^k\ot g^{n+i+l-k}\right\}}.
	\end{array}
	$$
	These two expressions for $\de\bigl(g^k\bigr)$ coincide. The same type of argument proves
	coassociativity of the coproduct.
	
\subsection{General semisimple algebras}

To forma general semisimple algebra it is enough to consider the following situation:

\begin{rem}
	Assume that $A_1$ and $A_2$ are $\Bbbk$-algebras. The product of the algebras $A_1$ and $A_2$ is the algebra $A = A_1 \times A_2$ with the addition and the
	multiplication given by the formulas $(a_1, a_2) + (b_1, b_2) = (a_1 + b_1, a_2 + b_2)$ and $(a_1, a_2)(b_1, b_2) = (a_1b_1, a_2b_2)$, where $a_1$, $b_1\in A_1$ and $a_2$, $b_2\in A_2$. The
	identity of $A$ is the element $1 = (1_{A_1}, 1_{A_2}) = e_1+e_2 \in A_1\times A_2$, where $e_1 = (1_{A_1}, 0)$ and $e_2 = (0, 1_{A_2})$. If $\bigl(A_1,\de_1\bigr)$ and $\bigl(A_2,\de_2\bigr)$ are nearly Frobenius algebras then $A$ admits a natural structure of Nearly Frobenius algebra. In the next paragraph we describe this structure.

	First, we define $\de(e_1)=\sum(a_1,0)\otimes (a_2,0)$, where $\de_1(1_{\mathcal{A}_1})=\sum a_1\otimes a_2$ and $\de(e_2)=\sum (0,b_1)\otimes (0,b_2)$, where $\de_2(1_{\mathcal{A}_2})=\sum b_1\otimes b_2$. Then
	$$\de(1)=\sum(a_1,0)\otimes (a_2,0)+\sum (0,b_1)\otimes (0,b_2)\in  	A\otimes A.$$
	To prove that this defines a bimodule morphism it is necessary to guarantee  that $\de(1)$ satisfies that
	$$(c\otimes 1)\de(1)=\de(1)(1\otimes c),\;\forall\; c\in A.$$  
	Denote $c=(c_1,c_2)\in A$, then
	$$\begin{array}{ccl}
	(c\otimes 1)\de(1)&=&\displaystyle{(c_1,c_2)\otimes(1,1)\left[\sum(a_1,0)\otimes (a_2,0)+\sum (0,b_1)\otimes (0,b_2)\right]}\\
	&=&\displaystyle{\sum\left((c_1,c_2)\otimes(1,1)\right)\left((a_1,0)\otimes (a_2,0)\right)+\sum \left((c_1,c_2)\otimes(1,1)\right)\left((0,b_1)\otimes (0,b_2)\right)}\\
	&=& \displaystyle{\sum (c_1a_1,0)\otimes (a_2,0)+\sum (0,c_2b_1)\otimes(0,b_2). }
	\end{array}$$
	
	On the other hand
	
	$$\begin{array}{ccl}
	\de(1)(1\otimes c)&=&\displaystyle{\left[\sum(a_1,0)\otimes (a_2,0)+\sum (0,b_1)\otimes (0,b_2)\right]\left((1,1)\otimes(c_1,c_2)\right)}\\
	&=&\displaystyle{\sum\left((a_1,0)\otimes (a_2,0)\right)\left((1,1)\otimes(c_1,c_2)\right)+\sum \left((0,b_1)\otimes (0,b_2)\right)\left((1,1)\otimes(c_1,c_2)\right)}\\
	&=& \displaystyle{\sum (a_1,0)\otimes (a_2c_1,0)+\sum (0,b_1)\otimes(0,b_2c_2). }
	\end{array}$$
	
	Remember that $\de_{A_1}$ and $\de_{A_2}$ are bimodule morphisms, then 
	$$(c_1\otimes 1)\de_{A_1}(1_{A_1})=\sum c_1a_1\otimes a_2=\sum a_1\otimes a_2c_1=\de_{A_1}(1_{A_1})(1\otimes c_1)$$
	and
	$$(c_2\otimes 1)\de_{A_2}(1_{A_2})=\sum c_2b_1\otimes b_2=\sum b_1\otimes b_2c_2=\de_{A_2}(1_{A_2})(1\otimes c_2)$$
	
	This proves that $(c\otimes 1)\de(1)=\de(1)(1\otimes c)$. Then $A$ is a nearly Frobenius algebra.
	
\end{rem}

\begin{coro}
	If $char(\k)$ does not divide the order of $G$, then $\k[G]$ is a
	nearly-Frobenius algebra.
\end{coro}
\begin{proof}
	Applying Maschke's theorem, we have that $\k[G]$ is semisimple,
	then it is a product of simple algebras $M_{n_i\times n_i}(\k)$.
	Therefore, by the Theorem \ref{theorem1}, we conclude that $\k[G]$
	is a nearly-Frobenius algebra. Moreover, we can determine all the
	nearly-Frobenius structures that it admits.
\end{proof}

From what we have seen, we conclude that, in the case of semi-simple algebras, the Frobenius space of $A$ is a vector space of dimension equal to the dimension of $A$, and that it has a one dimensional subspace (minus the origin) of \emph{bona fide} Frobenius structures.

\section{Nearly Frobenius Quiver Algebras} 

Quiver algebras provide a large collection of  examples of nearly-Frobenius algebras as have been shown in  \cite{AGL12} by Artenstein, Lanzilotta and the first author of this paper.
Let us summarize briefly these
results (see \cite{ASS06}).

First, recall that a \emph{quiver} $Q = \bigl(Q_0,Q_1, s, t\bigr)$ consists of two sets: $Q_0$ (whose elements are called
\emph{vertices}) and $Q_1$ (whose elements are
called \emph{arrows}), and two maps $s, t : Q_1 \rt Q_0$, which
associate to each arrow $\al\in Q_1$ its \emph{source} $s(\al)\in
Q_0$ and its \emph{target} $t(\al)\in Q_0$, respectively.

\begin{defn}
	Let $Q = \bigl(Q_0,Q_1, s, t\bigr)$ be a quiver and $a, b\in Q_0$. A
	\emph{path} of \emph{length} $l\geq 1$ with source $a$ and target
	$b$ (or, more briefly, from $a$ to $b$) is a sequence
	$$\bigl(a |\al_1,\al_2,\dots,\al_l| b\bigr),$$
	where $\al_k\in Q_1$ for all $1\leq k\leq l$, and we have
	$s\bigl(\al_1\bigr)=a$, $t\bigl(\al_k\bigr)=s\bigl(\al_{k+1}\bigr)$
	for each $1\leq k<l$, and finally $t\bigl(\al_l\bigr)=b$. Such a
	path is denoted briefly by $\al_1\al_2\dots\al_l$.
\end{defn}

\begin{defn}
	Let $Q$ be a quiver. The \emph{path algebra} $\k Q$ is the
	$\k$-algebra whose underlying $\k$-vector space has as its basis the
	set of all paths $\bigl(a|\al_1,\al_2,\dots,\al_l|b\bigr)$ of length
	$l\geq 0$ in $Q$ and such that the product of two basis vectors
	$\bigl(a|\al_1,\al_2,\dots,\al_l|b\bigr)$ and $\bigl(c
	|\be_1,\be_2,\dots,\be_k|d\bigr)$ of $\k Q$ is defined by:
	$$\bigl(a|\al_1,\al_2,\dots,\al_l|b\bigr)\bigl(c |\be_1,\be_2,\dots,\be_k|d\bigr)=\delta_{bc}\bigl(a|\al_1,\dots,\al_l,\be_1,\dots,\be_k|d),$$
	where $\delta_{bc}$ denotes the Kronecker delta. In other words, the
	product of two paths $\al_1\dots\al_l$ and $\be_1\dots\be_k$ is
	equal to zero if $t\bigl(\al_l\bigr) \neq s\bigl(\be_1\bigr)$ and is
	equal to the composed path $\al_1\dots\al_l\be_1\dots\be_k$ if
	$t\bigl(\al_l\bigr) = s\bigl(\be_1\bigr)$. The product of basis
	elements is then extended to arbitrary elements of $\k Q$ by
	distributivity.
\end{defn}
\begin{exmp}
	If $Q$ is the following quiver:
	$$\includegraphics{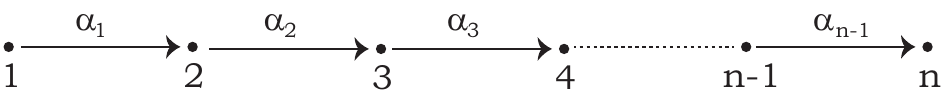}$$
	Then the path algebra $A=\k Q$  $$\k Q=\langle e_1, e_2,\dots
	,e_n,\al_i\dots \al_{i+j}:i=1,\dots,n, j\geq 0 \rangle.$$ admits a
	unique nearly-Frobenius structure, where the coproduct is defines as
	follows:
	$$\begin{array}{rcl}
	\de(e_1)&=&a\al_1\dots\al_{n-1}\ot e_1, \\
	\de(e_n)&=&a e_n\ot\al_1\dots\al_{n-1}, \\
	\de(e_i)&=&a\al_i\dots\al_{n-1}\ot\al_1\dots\al_{i-1}, \\
	\de(\al_i\dots\al_j)&=&a\al_i\dots\al_{n-1}\ot\al_1\dots\al_j,
	\end{array}
	$$ where $a\in\k.$
\end{exmp}

\begin{theorem}
	Let $A=\k Q$ with $Q$ a finite, connected quiver with no oriented
	cycles. Then $A$ has a nearly-Frobenius structure if and only if
	$Q=A_{n}$ with all the arrows in $Q$ having the same orientation.
\end{theorem}
If we introduce relations in the quiver $Q$, then the nearly-Frobenius structures over $Q$ are very interesting.
\begin{prop}
	The path algebra associated to the quiver
	$$\includegraphics{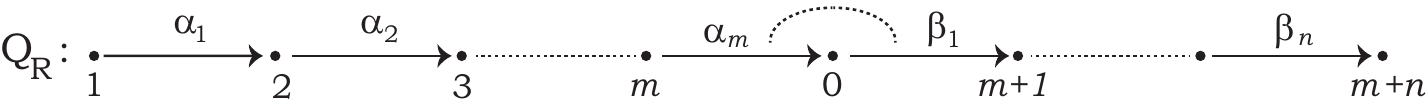}$$
	with the relation $\al_m\be_1=0$, admits $mn+2$ independent
	nearly-Frobenius structures; these are:
	$$\begin{array}{rclcrcl}
	\de\bigl(e_1\bigr) & =      & a\al_1\dots\al_m\ot e_1                &  & \de\bigl(e_{m+1}\bigr) & =      & b\be_2\dots\be_n\ot\be_1 \\
	& \vdots &                                        &  &                        & \vdots &  \\
	\de\bigl(e_i\bigr) & =      & a\al_i\dots\al_m\ot\al_1\dots\al_{i-1} &  & \de\bigl(e_{m+i}\bigr) & =      & b\be_{i+1}\dots\be_n\ot\be_1\dots\be_i \\
	& \vdots &                                        &  &                        & \vdots &  \\
	\de\bigl(e_m\bigr) & =      & a\al_m\ot\al_1\dots\al_{m-1}           &  & \de\bigl(e_{m+n}\bigr) & =      & be_{m+n}\ot\be_1\dots\be_n \\
	\de\bigl(\al_i\dots\al_j\bigr) & = & a\al_i\dots\al_m\ot\al_1\dots\al_j  &  & \de\bigl(\be_i\dots\be_j\bigr) & = & b\be_i\dots\be_n\ot\be_1\dots\be_j
	\end{array}$$
	$$\de\bigl(e_0\bigr)=ae_0\ot\al_1\dots\al_m+b\be_1\dots\be_n\ot e_0+\sum_{i=1}^m\sum_{j=1}^nc_{ij}\be_1\dots\be_j\ot\al_i\dots\al_m$$
	where $a,b, c_{ij}\in\k$.
\end{prop}
\begin{theorem}
	The path algebra $A$ associated to the cyclic quiver $Q$
	$$\includegraphics{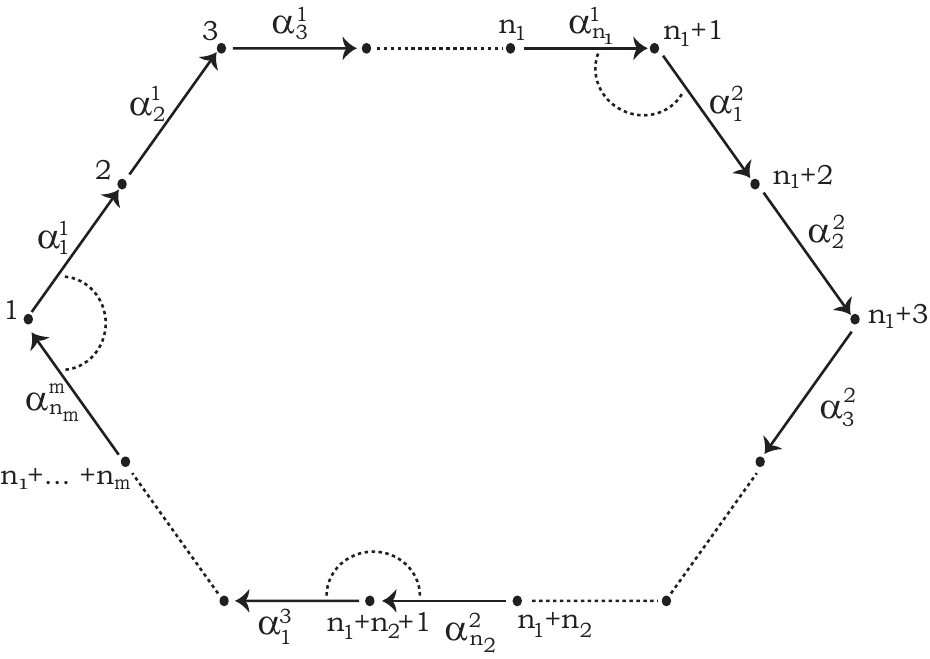}$$
	with $m$ maximal paths of length $n_i$, $i=1,\dots, m$ admits $R$
	nearly-Frobenius structures, where
	$$R=m+\sum_{i=1}^mn_in_{i+1}$$
	with $n_{m+1}=m_1.$
\end{theorem}

\section{Almost TQFTs}

Just as Frobenius algebras correspond to TQFTs, nearly Frobenius algebras correspond to \emph{almost TQFTs}.

\begin{defn}
	An \emph{almost TQFT} is a functor $Z$ from the full subcategory $\mathbf{Cob}_2^+$ of $\mathbf{Cob}_2$, whose objects are positive integers $n>0$ (standing, as before, for the disjoint union of $n$ circles), into the category $\mathbf{Vect}^\infty$ of (possibly infinite dimensional) vector spaces.
\end{defn} 

It could be the case that an almost TQFT is also a TQFT, but this happens rarely. Almost TQFTs often have no traces, but they always have coproducts.

\begin{theorem}
	The category of nearly Frobenius algebras is equivalent to the category of almost TQFTs.
\end{theorem}

\begin{proof}
	Notice that the way we set up the proof that $F$ in equation \ref{Folk} works as well to prove this theorem. Again, given an almost TQFT it is easy to define on $A:=Z(S^1)$ a nearly Frobenius algebra structure. Conversely, the construction of the almost TQFT from a nearly Frobenius algebra works by using an arbitrary $f$ in $\mathcal{M}(\Sigma)$ to produce a pair-of-pants decomposition, and the independence of the decomposition follows from Cerf's theory.
\end{proof}

\section{String Topology}

The following is a fundamental example: The cohomology of a \emph{compact} manifold is always a Frobenius algebra, but this fails to be the case for non-compact manifolds.

\begin{exmp}
	The Poincar\'e algebra $A:=H^*(M)$ of a non-compact manifold  $M$ admits
	a nearly Frobenius algebra structure induced from the smooth structure in $M$. Consider the diagram: \begin{equation*}
	\begin{split}
	\xymatrix{
		M\ar[r]^{\de}\ar[d]_{\de}&M\ti M\ar[d]^{1\ti\de}\\
		M\ti M\ar[r]_{\de\ti 1}&M\ti M\ti M }
	\end{split}
	\end{equation*}
	From transversality it follows that:
	$$(\Delta\ti 1)^*(1\ti \Delta)^!=\Delta^!\Delta^*,$$
	where $\de^*:\h^*(M)\ot \h^*(M)=\h^*(M\ti M)\rt \h^*(M)$ is the map
	induced by the diagonal map in cohomology, and $\de^!:\h^*(M)\rt
	\h^*(M)\ot \h^*(M)$ is the Gysin map for the diagonal map. Then:
	$$\bigl(\de^*\ot 1\bigr)\bigl(1\ot \de^!\bigr)=\de^!\de^*.$$
	Therefore  $\h^*(M)$ is an algebra with a coproduct which is a module homomorphism.  \\
	Non-compact manifolds lack a
	fundamental class in homology, and so we don't have a trace in cohomology. The coproduct $\Delta$ is a substitute for Poincar\'{e} duality in this context: it plays the role of the Poincar\'{e} dual for the cup product.
\end{exmp}
 
For an $n$-dimensional compact manifold $M$, the free loop space is the mapping space $LM:=C^0(S^1,M)$. Chas and Sullivan in \cite{ChasSullivan} used the intersection product to define an intersection product of the form:
$$\bullet:\h_l(LM)\otimes \h_m(LM) \to \h_{l+m-n}(LM).$$
This is defined as follows: given singular simplices of loops $\tilde{\sigma}_1 \in C_l(LM)$ and $\tilde{\sigma}_2 \in C_m(LM)$ we can evaluate each loop at zero, to obtain two singular simplices $\sigma_1 \in C_l(M)$ and $\sigma_2 \in C_m(M)$, then (perhaps using transversality) we can intersect both simplices to define a singular chain $\sigma_3:=\sigma_1 \cap \sigma_2 \in C_{1+m-n}(M)$. At every point $p \in \sigma_3$ we can concatenate the loop $\tilde{\sigma}_1(p) \in LM$ with the loop  $\tilde{\sigma}_2(p) \in LM$ in that order (both loops pass through $p$ at time $0$), to obtain a loop $\tilde{\sigma}_3(p)$,  thus defining a singular chain of loops $\tilde{\sigma}_3  \in C_{l+m-n}(LM)$. This definition works even if $M$ is non-compact  (see, for example, the paper by Cohen and Jones \cite{CohenJones} where they make this definition rigurous).

The point for us  is that even though $\h_*(LM)$ is almost never a Frobenius algebra, it always is a nearly Frobenius algebra as the following theorem by Cohen and Godin \cite{CohenGodin} states:

\begin{theorem}
	The Chas-Sullivan algebra $\h_*(LM)$ admits a nearly Frobenius algebra structure (induced by the smooth structure of $M$) whose coproduct extends that of $(\h_*(M), \Delta)$.
\end{theorem}

In particular $\h_*(LM)$ is the space state of an almost TQFT (whose structure is induced by the smooth structure of $M$).

This defines a functor from the category of smooth manifolds to the category of nearly Frobenius algebras. While the Chas-Sullivan algebra happens to be homotopy invariant (see \cite{Cohen2008homotopy}), the nearly Frobenius algebra is not. Whether the coproduct depends on the diffeomorphism type or only on the homeomorphism type is an interesting question.

\section{The moduli variety $\mathcal{N}(A)$ of nearly Frobenius structures on an algebra $A$}

Let $A$ be an algebra, the Frobenius identities make the set of all possible nearly Frobenious coproducts $\mathcal{N}(A)$ on $A$ into a possibly infinite dimensional algebraic variety over $\k$.
 
The following fact is somewhat surprising:

\begin{theorem}
	For a $\k$-algebra $A$, the variety $\mathcal{N}(A)$ of nearly
	Frobenius coproducts of $A$ making it into a nearly Frobenius algebra is a linear $\k$-vector space.
\end{theorem}

	 First, the category of $A$ bimodules will be written as ${}_A\mathcal{M}_A$. For an object $M\in {}_A\mathcal{M}_A$ we write
$$I(M) =\bigl\{m\in M: a\cdot m=m\cdot a\;\forall a\in A\bigr\}$$
to denote the sub–bimodule of invariants.

The proof consists of two lemmas.

\begin{lemma}\label{lema:ea} For an arbitrary $\Bbbk$--algebra $A$, the map ${\tt e}: \mathcal E_A \to I(A \otimes A)$ defined as ${\tt e}(\Delta)=\Delta(1)$, is a bijection.                                                            \end{lemma}
\begin{proof} First observe that the map ${\tt e}$ makes sense. We have that $\Delta(x)=\Delta (x1)=\Delta(1x)=(x \otimes 1)\Delta(1)=\Delta(1)(1 \otimes x)$ for all $x \in A$. This shows that the map ${\tt e}$ is injective and also its codomain is $I(A \otimes A)$.
	Given an element $\xi=\sum a_i \otimes b_i \in I(A \otimes A)$, if
	we define $\Delta_\xi(x)=x \cdot \xi = \xi \cdot x$ it is clear that
	$\Delta_\xi$ is a nearly Frobenius structure in $A$ and that ${\tt
		e}(\Delta_\xi)=\xi$.  We check that it is nearly Frobenius, for
	example: $\Delta_\xi(xy)=(xy)\cdot \xi=x \cdot \Delta_\xi(y)$ and
	similarly for the action on the left.  As to the coassociativity of
	$\Delta$ first we observe that if the element $\Delta(1)=\sum a_i
	\otimes b_i$ it is clear that: $(\Delta \otimes \id)\Delta(1)=\sum
	\Delta(a_i)\otimes b_i=\sum a_j \otimes b_ja_i \otimes b_i=\sum a_j
	\otimes \Delta(b_j)= (\id \otimes \Delta)(\Delta(1))$. The
	coassociativity for a general $x$ follows from the basic property
	$\Delta(x)=x\cdot \Delta(1)=\Delta (1) \cdot x$.
\end{proof}

\begin{lemma} The set $\mathcal N_A$ of nearly Frobenius structures in $A$ is a vector space. Moreover, $\dim \mathcal N_A \leq (\dim A)^2$. 
\end{lemma}
\begin{proof} Clearly, once the above bijection is established, we can induce from the linear structure of $I(A \otimes A)$ a linear structure in $\mathcal N_A$: the structure that is induced is the sum in the space of linear maps from $A$ into $A \otimes A$.  The bound of the dimensions is innediate.
\end{proof}

One final remark: for most simple examples we have:
$$\Frobdim(A):= \dim \mathcal{N}(A) = \dim_\k (A),$$
but, as the examples of quiver algebras show, this is not always tha case. In any case the previous lemma provides a bound. This suggests that an investigation of the meaning of the invariant $\Frobdim(A)$ is a worthwhile question.

\section{Acknowledgments}

The second author would like to  thank the Moshinsky Foundation, Conacyt, FORDECYT-265667, the Samuel Gitler International Collaboration Center, the Laboratory of Mirror Symmetry NRU HSE, RF Government grant, ag. No. 14.641.31.0001 and the kind hospitality of the Universities of Geneva and of Miami.

We would like to thank the referee for very careful and useful remarks that improved this paper.



\end{document}